\documentclass[11pt,letterpaper]{amsart}
\usepackage[T1]{fontenc}
\usepackage{lmodern}
\usepackage{amsmath,amssymb,mathtools}
\usepackage[margin=1.15in]{geometry}
\usepackage{microtype}
\usepackage{enumitem}
\usepackage{xcolor}
\definecolor{linkblue}{RGB}{25,55,95}
\usepackage[colorlinks=true,linkcolor=linkblue,citecolor=linkblue,urlcolor=linkblue,pdfusetitle]{hyperref}
\setlist[enumerate]{label=\textup{(\arabic*)},leftmargin=*,itemsep=2pt}
\numberwithin{equation}{section}
\newtheorem{theorem}{Theorem}[section]
\newtheorem{proposition}[theorem]{Proposition}
\newtheorem{lemma}[theorem]{Lemma}
\newtheorem{corollary}[theorem]{Corollary}
\theoremstyle{definition}

\theoremstyle{remark}
\newtheorem{remark}[theorem]{Remark}
\newcommand{\Z}{\mathbb Z}
\newcommand{\Q}{\mathbb Q}
\newcommand{\C}{\mathbb C}
\newcommand{\R}{\mathbb R}
\newcommand{\F}{\mathbb F}
\newcommand{\Zp}{\Z_p}
\newcommand{\Qp}{\Q_p}
\newcommand{\Fp}{\F_p}
\newcommand{\OO}{\mathcal O}
\newcommand{\cP}{\mathcal P}

\newcommand{\Abar}{\overline A}
\newcommand{\Bbar}{\overline B}
\newcommand{\m}{\mathfrak m}
\newcommand{\pt}{\mathrm{pt}}
\newcommand{\BM}{\mathrm{BM}}
\newcommand{\tors}{\operatorname{tors}}
\newcommand{\length}{\operatorname{length}}
\newcommand{\Sym}{\operatorname{Sym}}
\newcommand{\Frac}{\operatorname{Frac}}

\newcommand{\depth}{\operatorname{depth}}
\newcommand{\pd}{\operatorname{pd}}
\newcommand{\rank}{\operatorname{rank}}
\newcommand{\End}{\operatorname{End}}

\title[Equivariant multiplicities and torsion]{Equivariant multiplicities and torsion \\
at attractive fixed points}

\author{Tao Gui}
\address{(Tao Gui) \newline \indent Institute for Theoretical Sciences, Westlake University, No.\ 600 Dunyu Road, Xihu District, Hangzhou, Zhejiang Province 310030, China}
\email{guitao18(at)mails(dot)ucas(dot)ac(dot)cn}

\author{Peter L. Guo}
\address{(Peter L. Guo) \newline \indent Center for Combinatorics, Nankai University, LPMC, Tianjin 300071, People’s Republic of China}
\email{lguo(at)nankai(dot)edu(dot)cn}

\author{Zhuowei Lin}
\address{(Zhuowei Lin) \newline \indent Center for Combinatorics, Nankai University, LPMC, Tianjin 300071, People’s Republic of China}
\email{zwlin0825(at)163(dot)com}

\date{}
\hypersetup{pdftitle={Equivariant multiplicities and torsion at attractive fixed points}}

\begin{document}

\begin{abstract}
Let \(X\) be a rationally smooth complex affine variety with a torus action and an attractive fixed point \(x\). Suppose that \(X\setminus\{x\}\) is \(p\)-smooth and that its integral equivariant cohomology has no \(p\)-torsion. We prove that the order of its total \(p\)-primary cohomology torsion is the \(p\)-part of the reduced numerator of the equivariant multiplicity at \(x\). In particular, this resolves the torsion-order conjecture of Juteau--Williamson and its local version for normal slices in Schubert varieties, using the equivariant torsion-freeness result of Fiebig--Williamson.
\end{abstract}
\maketitle

\section{Introduction}

Let \(X\) be an irreducible $n$-dimensional complex algebraic variety equipped with its classical topology, and let \(p\) be a prime. A point \(x\in X\) is \emph{\(p\)-smooth} if
\[
 H^\bullet(X,X\setminus\{x\};\Fp) \cong H^\bullet(\C^n,\C^n\setminus\{0\};\Fp).
\]
The \emph{\(p\)-smooth locus} of X is the largest open subset consisting of $p$-smooth points. Replacing \(\Fp\) by \(\Q\) or \(\Z\) gives rational smoothness (\(\Q\)-smoothness) and integral smoothness (\(\Z\)-smoothness). As is clear from the definition, these are local topological conditions and it is not easy to decide whether a point \(x\in X\) is $k$-smooth (for $k=\Z, \Q, \Fp)$ in general. 

For varieties with an algebraic torus $T$-action, equivariant multiplicities provide a way of studying $k$-smoothness of $T$-fixed points algebraically. Kumar's criterion detects the smooth and rationally smooth loci of Schubert varieties using equivariant multiplicities \cite{Kumar}. Arabia \cite{Arabia} and Brion \cite{Brion} generalized Kumar's criterion to more general $T$-varieties.

Juteau and Williamson  \cite[Theorem 1.2]{JW} gave a corresponding criterion using the primes dividing the numerator of the equivariant multiplicity to determine \(p\)-smoothness, building on the Kumar--Arabia--Brion criterion for rational smoothness. In particular, their result \cite[Theorem 1.3]{JW} exhibits a combinatorial criterion for the determination of the \(p\)-smooth locus of Schubert varieties for all primes $p$. 

In \cite[Conjecture~1.5]{JW}, Juteau and Williamson predicted that  the entire numerator of the equivariant multiplicity has a topological interpretation. Specifically, \cite[Conjecture 1.5]{JW} states that  in a Schubert variety, for a rationally smooth normal slice $N$ which is smooth away from its distinguished $T$-fixed point $x$,   the numerator $f_x$ of the equivariant multiplicity is equal to the order of the torsion in the integral cohomology $H^\bullet(U;\Z)$ of the punctured slice $U:=N\setminus\{x\}$. Below \cite[Conjecture 1.5]{JW}, they also proposed a $p$-local version of their conjecture as well as an extension to affine $T$-varieties. As a consequence, the numerator of the equivariant multiplicity in these cases is in fact a topological invariant (and not just an invariant of the singularity with $T$-action). Our goal in this paper is to provide proofs of these statements.

Let $R$ be a commutative domain. Let \(T\cong(\C^\times)^r\) be a complex torus, and \(\Lambda=\chi^*(T)\) its character lattice. Write 
\[
S_R:=\Sym_R(\Lambda\otimes_{\Z}R),\quad \text{ with } \deg\Lambda=2,
\]
for the symmetric algebra. Denote by $Q_R$  the fraction field of $S_R$.

Assume that \(X\) is a complex affine irreducible $T$-variety with an attractive (hence unique) $T$-fixed point \(x\). Here attractive means that a cocharacter of \(T\) contracts \(X\) to \(x\). Set \(U=X\setminus\{x\}\) and \(n=\dim_{\C}X\). If \(X\) is rationally smooth, by the Kumar--Arabia--Brion criterion, the \emph{equivariant multiplicity}  $e_xX\in Q_\Z$ of $x$ has a reduced expression (see \cite[(1.1)]{JW})
\begin{equation}\label{eq:intro-mult}
 e_xX=\frac{d}{\pi},\quad \text{ with } \pi=\chi_1\cdots\chi_n,\quad \chi_i\in\Lambda\setminus\{0\},\quad d\in\Z\setminus\{0\}.
\end{equation}
Here the reducedness is taken in \(S_{\Z}\): no prime $p$ divides both \(d\) and \(\pi\). For the reader's convenience, we will give a proof of the existence of \eqref{eq:intro-mult} in Lemma~\ref{lem:multiplicity}.

For a graded module \(M\) over a domain \(R\), write \(\tors_R M\) for the direct sum of its \(R\)-torsion submodules in all degrees. Throughout, \(\Zp\) denotes the ring of \(p\)-adic integers, with the \(p\)-adic valuation \(v_p\). Our first main result  confirms a conjecture of Juteau and Williamson, stated after \cite[Conjecture 1.5]{JW} as an extension of the local version of the conjecture on Schubert slices to affine varieties.

\begin{theorem}\label{thm:affine}
Let \(X\) be an irreducible, rationally smooth complex affine \(T\)-variety of dimension \(n>0\), with an attractive (hence unique) fixed point \(x\). Assume that \(U:=X\setminus\{x\}\) is \(p\)-smooth and that \(H_T^\bullet(U;\Z)\) is free of \(p\)-torsion. With the normalization \eqref{eq:intro-mult}, one has
\begin{equation}\label{eq:affine-order}
\bigl|\tors_{\Zp}H^\bullet(U;\Zp)\bigr|=p^{v_p(d)}.
\end{equation}
More precisely, there are nonnegative integers \(a_1,\ldots,a_{n-1}\) such that
\begin{equation}\label{eq:degreewise}
 H^j(U;\Zp)\cong
 \begin{cases}
 \Zp & j=0,\,2n-1,\\[5pt]
 \Zp/(p^{a_i}) & j=2i,\quad 1\leq i<n,\\[5pt]
 0 & \text{otherwise}.
 \end{cases}
 \qquad \sum_{i=1}^{n-1}a_i=v_p(d),
\end{equation}
Here \(\Zp/(p^0)\) is the zero module.
\end{theorem}

As a consequence, we prove the following global version of the Juteau--Williamson conjecture.

\begin{corollary}\label{cor:affine-integral}
In the situation of Theorem~\ref{thm:affine}, assume instead that \(U\) is $\Z$-smooth and \(H_T^\bullet(U;\Z)\) is a free $\Z$-module. Then
\begin{equation}\label{eq:integral-order}
\bigl|\tors_{\Z}H^\bullet(U;\Z)\bigr|=|d|.
\end{equation}
Moreover, every intermediate even cohomology group is finite cyclic, and every intermediate odd cohomology group is zero.
\end{corollary}

\vspace{10pt}

Now let \(G\) be a connected complex reductive algebraic group, with a Borel subgroup \(B\) and maximal torus $T$, and let \(W\) be its Weyl group. Write
\[
 C_w=BwB/B,\qquad X_w=\overline{C_w},
\]
for the Schubert cell and Schubert variety corresponding to $w\in W$. For \(v\leq w\) in the Bruhat order, let \(N\) be the standard $T$-invariant affine normal slice to \(C_v\) in \(X_w\), with distinguished $T$-fixed point \(x_v:=vB/B\). For later convenience, we take $T$ to be the torus with its character lattice being the root-lattice \(\Z\Phi\), or equivalently, we take $G$ to be the adjoint group with maximal torus $T$. When \(N\) is rationally smooth, by Kumar's criterion \cite{Kumar}, we can denote by \(f_{v,w}\) the (positive) constant numerator appearing in the reduced fraction of the equivariant multiplicity \(e_{x_v}X_w\). The integral formula \eqref{eq:integral} of the following theorem is conjectured (when \(U\) is smooth) in \cite[Conjecture~1.5]{JW} (with the $p$-local version in the paragraph after \cite[Conjecture~1.5]{JW}).

\begin{theorem}\label{thm:schubert}
Suppose that \(N\) is rationally smooth and \(U:=N\setminus\{x_v\}\) is \(p\)-smooth. Then
\[
\bigl|\tors_{\Zp}H^\bullet(U;\Zp)\bigr|=p^{v_p(f_{v,w})}.
\]
For \(n=\dim_{\C}N>0\), the description \eqref{eq:degreewise} holds with \(v_p(f_{v,w})\) in place of \(v_p(d)\). If \(U\) is $\Z$-smooth, then
\begin{equation} \label{eq:integral}
\bigl|\tors_{\Z}H^\bullet(U;\Z)\bigr|=f_{v,w}.
\end{equation}
\end{theorem}

\vspace{10pt}

Theorem \ref{thm:schubert} also holds for Schubert varieties in Kac--Moody flag varieties.

Theorem \ref{thm:schubert} is an application of Theorem~\ref{thm:affine} and Corollary~\ref{cor:affine-integral}, by borrowing  the equivariant torsion-freeness theorem for $H_T^\bullet(N\setminus\{x_v\};\Zp)$ of Fiebig and Williamson \cite[Corollary~8.10]{FW}; the hypotheses and the normalization are explained in Section~\ref{sec:schubert}. 

The proof of Theorem \ref{thm:affine} needs the following structural theorem which does not require rational smoothness at the attractive point.

\begin{theorem}\label{thm:pd}
Let \(X\) be an irreducible complex affine \(T\)-variety of positive dimension with an attractive fixed point \(x\). Put \(U:=X\setminus\{x\}\), \(k=\Fp\), and \(P:=S_k=H_T^\bullet(\pt;k)\). If \(U\) is \(p\)-smooth and \(H_T^\bullet(U;k)\) is torsion over \(P\), then we have the projective dimension 
\[
 \pd_P H_T^\bullet(U;k)=1.
\]
In particular, this module is Cohen--Macaulay of dimension \(r-1\), where $r$ is the rank of the algebraic torus $T$.
\end{theorem}

The paper is organized as follows. In Section \ref{sec:preliminaries}, we fix our conventions on the equivariant cohomology and equivariant multiplicities. In Section \ref{sec:geometry}, we give a proof of Theorem \ref{thm:pd}. In the setting of Theorem \ref{thm:affine}, Theorem \ref{thm:pd} gives the regularity needed to specialize over a finite unramified extension of \(\Zp\) to a line. The specialization argument is contained in Section \ref{sec:specialization}. In Section \ref{sec:ordinary}, we provide a proof of Theorem \ref{thm:affine} and Corollary \ref{cor:affine-integral}, and record a consequence on duality when the punctured $U$ is smooth. Section \ref{sec:schubert} is devoted to  the Schubert application and a shorter parity-sheaf proof of its regularity assertion. 

\subsection*{Acknowledgments}
This research was done while the first author was visiting the Center for Combinatorics at Nankai University and he thanks the center for the hospitality. The first author is supported in part by the National Natural Science Foundation of China (No. 12471309). The second author is supported by the National Natural Science Foundation of China (No. 12371329) and the Fundamental Research Funds for the Central Universities (No. 63263094).

\subsection*{AI disclosure}
Generative AI was used interactively in the mathematical exploration, including the development, testing, and refinement of parts of the proof, and as a technical assistant for checking calculations and arguments, locating relevant references, improving the exposition and identifying mathematical and typographical errors. All AI-assisted content was carefully verified by the authors, who take full responsibility for the manuscript.

\section{Equivariant cohomology and multiplicities}\label{sec:preliminaries}

In this paper, all varieties $X$ will be reduced complex algebraic varieties, endowed with their classical (metric) topology. For an algebraic torus \(T\), its maximal compact subtorus \(K\) gives the same Borel equivariant cohomology. If $R$ denotes a ring of coefficients (usually $\Z, \Fp , \Zp$, or $\Qp$), then we write \(D_T(X;R)\) for the equivariant constructible derived category of the $T$-variety $X$, and \(\omega_{X,R}\) for its (both equivariant and nonequivariant by abuse of notation) dualizing complex, with the conventions in \cite{BL,JW}. Thus
\[
 H^{-m}(X;\omega_{X,R})=H_{m}^{\BM}(X;R)
\]
is the Borel--Moore homology of a locally closed space $X$. On a complex \(R\)-homology manifold $X$ of complex dimension \(n\), the complex orientation identifies \(\omega_{X,R}\) with the constant sheaf \(R_X[2n]\). Our grading convention is \(M[a]^j=M^{j+a}\).

For a complex algebraic variety, \(p\)-smoothness gives the orientation isomorphism over both \(\Fp\) and \(\Zp\), and hence over every finite unramified extension of \(\Zp\) (that is, the ring of integers of a finite unramified extension of $\Qp$). One can see this from the local universal coefficient sequence: finite generation of local integral cohomology and Nakayama's lemma lift the local \(\Fp\)-homology-manifold condition to \(\Zp\). The complex orientation trivializes the resulting rank-one orientation local system.

\subsection{Attractive fixed points and their links}
An affine $T$-variety $X$ with an attractive $T$-fixed point x can be contracted to the point \(x\) by a cocharacter \(\lambda\) of $T$; hence there is a closed $T$-equivariant embedding \(X\hookrightarrow V\) taking \(x\) to the origin of the affine space $V$, with all \(\lambda\)-weights positive. Indeed, choosing finitely many homogeneous generators of its positively graded coordinate algebra does the job. Write the positive weights as \(a_1,\ldots,a_m\). For the coordinates $z_j$, $j=1, \dots, m$, of $V$, set
\[
 h(z)=\sum_{j=1}^{m}|z_j|^2,\qquad L=X\cap h^{-1}(1).
\]
The squared Hermitian norm $h$ may be chosen \(K\)-invariant after averaging over $K$. We record the following lemma.

\begin{lemma}\label{lem:link}
Let \(X\) be an irreducible complex affine \(T\)-variety of positive dimension with an attractive fixed point \(x\). Put \(U:=X\setminus\{x\}\). Then the link \(L\) is a compact semialgebraic space of real dimension \(2n-1\), and
\[
 L\times\R_{>0}\longrightarrow U,\qquad (\ell,t)\longmapsto\lambda(t)\ell
\]
is a \(K\)-equivariant homeomorphism. The variety \(X\) is the open cone on \(L\), including its vertex. Consequently \(U\) has the homotopy type of a finite CW complex of dimension at most $2n - 1$, and 
\[
 H^j(U;R)=0\quad(j>2n-1),\qquad H_0^{\BM}(X;R)=0.
\]
\end{lemma}

\begin{proof}
On each $\lambda(0,\infty)$-orbit in \(U\),
\[
 h(\lambda(t)z)=\sum_j t^{2a_j}|z_j|^2
\]
is strictly increasing from zero to infinity. Thus each orbit meets \(L\) once. The radial parameter $t$ varies continuously, and positivity of the weights gives uniform convergence to the vertex as \(t\to0\) and uniform escape to infinity as \(t\to\infty\). This proves the product and cone descriptions. Semialgebraic triangulation gives the finite CW assertion and the cohomology dimension bound of $U$. The one-point compactification of \(X\) is the suspension of the nonempty space \(L\), so its degree-zero Borel--Moore homology is zero.
\end{proof}

In particular, ordinary cohomology of $U$ is finitely generated over the coefficient ring $R$ in each degree. The Leray--Serre spectral sequence of the Borel fibration then makes \(H_T^\bullet(U;R)\) finite over \(S_R=H_T^\bullet(\pt;R)\). In particular, all cohomology modules here are bounded below and degreewise finite for the coefficient rings. The same finiteness for the dualizing module below follows from the closed-open long exact sequence, see~\cite[(4.1) and (4.2)]{JW}.

\vspace{10pt}

\subsection{The reduced equivariant multiplicity}
For any irreducible complex algebraic variety $X$ of complex dimension $n$, the canonical complex orientation of the smooth locus $X^{\text{reg}}$ defines the \emph{equivariant fundamental class} (see \cite[p.2624]{JW})
\[
 \mu_X\in H_T^{-2n}(X;\omega_{X,\Z}).
\]
In the attractive situation, \(x\) is the unique fixed point. Following \cite[Definition 4.1]{JW}, the \emph{equivariant multiplicity} $e_xX\in Q_\Z=\Frac(S_{\Z})$ of $x\in X$ is defined by the point-pushforward
\[
 i_*(e_xX)=\mu_X,\qquad i:\{x\}\hookrightarrow X,
\]
understood after localization of both sides from the coefficient ring $S_{\Z}$ to the fraction field $Q_\Z$. This definition may also be made with the rational coefficients $\Q$ rather than the integral coefficients $\Z$.

\vspace{10pt}
As promised in the introduction, we give a proof of the existence of the reduced form \eqref{eq:intro-mult} of the equivariant multiplicity.

\begin{lemma}\label{lem:multiplicity}
Suppose that \(X\) is irreducible, affine and rationally smooth of dimension \(n>0\), with an attractive point \(x\). Then
\[
H_T^\bullet(X;\omega_{X,\Q})=S_{\Q}\mu_X, \qquad i_*(1)=g\mu_X,
\]
where \(g\in S_{\Q}\) is nonzero of degree \(2n\). As a consequence, \(e_xX=g^{-1}\). Moreover, \(g\) is a rational scalar times a product of \(n\) integral characters. Hence \eqref{eq:intro-mult} exists.
\end{lemma}

\begin{proof}
Rational smoothness identifies the dualizing complex $\omega_{X,\Q}$ with \(\Q_X[2n]\), and contraction identifies the $T$-equivariant cohomology of \(X\) with \(S_{\Q}=H_T^\bullet(pt;\Q)\). Thus the equivariant fundamental class $\mu_X$ is a generator of $H_T^\bullet(X;\omega_{X,\Q})$ in degree \(-2n\). By \cite[Lemma 4]{Brion} or \cite[Lemma 3.4]{FW} combined with standard arguments on the open/closed long exact sequence (see, e.g., \cite[p.2625]{JW}), localization after inverting nonzero characters makes the point-pushforward an isomorphism, since \(x\) is the only $T$-fixed point. It follows that \(g\ne0\), hence \(g\) becomes a unit in that localization, and that \(e_xX=g^{-1}\).

Unique factorization in \(S_{\Q}\) now shows that the irreducible factors of \(g\) are linear character factors. Choose primitive integral representatives \(\eta_1,\ldots,\eta_n\), with possible repetitions, and write
\[
 g=\frac{a}{b}\eta_1\cdots\eta_n, \qquad a,b\in\Z\setminus\{0\},\quad \gcd(a,b)=1.
\]
Absorb \(a\) into \(\eta_1\). Gauss's lemma for primitivity shows that \(b/(a\eta_1\cdots\eta_n)\) is reduced in \(S_{\Z}\), as required.
\end{proof}

For a smooth $T$-variety $X$, our conventions for $e_xX$ give the reciprocal of the product of the tangent weights of $T$ at $x\in X$, see also \cite[Remark 6.1]{JW}.

\vspace{10pt}

\section{Equivariant cohomology of projective dimension one}\label{sec:geometry}

In this section, we give a proof of Theorem~\ref{thm:pd}. Write \(P:=S_k=H_T^\bullet(\pt;k)\) for a field $k$. 

\subsection{Equivariant formality}
Recall that an action of $T$ on $X$ is \emph{equivariantly formal over \(k\)} if $H_T^\bullet(X;k)$ is a free $P$-module~\cite{GKM}. In particular, the forgetful map \(H_T^\bullet(X;k)\to H^\bullet(X;k)\) is onto. 

For a space $X$ of finite cohomological type, we denote the total Betti number over $k$ as
\[
 \beta_k(X)=\sum_j\dim_k H^j(X;k).
\]
We need the following consequence of the Bia\l ynicki-Birula decomposition.

\begin{lemma}\label{lem:bb}
Let \(X\) be a complex smooth projective variety with a \(\C^\times\)-action, and let \(F_\alpha\) be its connected fixed components. Then the attracting strata are affine-space bundles over the \(F_\alpha\), which gives a closed filtration of $X$. For every field \(k\),
\begin{equation}\label{eq:bb-betti}
\beta_k(X)=\sum_\alpha\beta_k(F_\alpha).
\end{equation}
\end{lemma}

\begin{proof}
The filtration and affine bundle structure follow from the Bia\l ynicki-Birula decomposition~\cite{BB}. Its integral Chow-correspondence form gives a direct-sum decomposition (the ``motivic Bia\l ynicki-Birula decomposition'') into Tate twists of the fixed components~\cite[Theorem~3.3 and Remark~2.1]{Brosnan}. Applying Betti cohomology with coefficients in \(k\) gives \eqref{eq:bb-betti}.
\end{proof}

\begin{remark}
The correspondences in this result have integral coefficients \cite[Remark~2.1]{Brosnan}. Thus \eqref{eq:bb-betti} holds even when the integral cohomology of a fixed component has torsion. 
\end{remark}

We need the following somewhat surprising result due to Bai and Pomerleano~\cite{BP}, which in characteristic $0$ is a consequence of Deligne’s proof that the Leray--Serre spectral sequence degenerates for smooth projective morphisms between complex algebraic varieties~\cite{Deligne}.

\begin{lemma}\label{lem:projective-formality}
Every algebraic torus action on a complex smooth projective variety is equivariantly formal over every field.
\end{lemma}

\begin{proof}
This is a special case of the more general Theorem~\cite[Theorem 1.3]{BP}.
\end{proof}

The above lemmas have the following corollary.

\begin{lemma}\label{lem:open-formality}
Let \(Y\) be a $T$-invariant open subvariety of a smooth projective \(T\)-variety \(X\). Suppose that a cocharacter \(\lambda\) has proper fixed locus in \(Y\), and that every point of \(Y\) has a limit in \(Y\) as \(t\to0\). Then
\[
H^\bullet(X;k)\longrightarrow H^\bullet(Y;k) \quad\text{and}\quad H_T^\bullet(Y;k)\longrightarrow H^\bullet(Y;k)
\]
are onto. In particular, \(H_T^\bullet(Y;k)\) is finite graded free over \(P\).
\end{lemma}

\begin{proof}
The open and proper subvariety \(Y^\lambda\subset X^\lambda\) is a union of connected components. Moreover, under the assumptions, \(Y\) is a union of entire attracting strata. 

Put \(D=X\setminus Y\). Take a closed Bia\l ynicki-Birula filtration \(X_i\) of $X$. The closed sets \(D\cap X_i\), followed by \(D\cup X_i\), give
\[
 \varnothing\subset D\cap X_1\subset \cdots \subset D \subset D \cup X_1\subset\cdots\subset X.
\]
Removing repetitions gives a closed filtration
\[
 \varnothing=Z_0\subset Z_1\subset\cdots\subset Z_s=X,
\]
such that each nonempty successive difference is a whole original stratum, and $D=Z_a$ for some $a$. Set \(W_i=X\setminus Z_i\), \(S_i=Z_i\setminus Z_{i-1}\). Each successive difference \(S_i\) is smooth and closed in the smooth open subvariety \(W_{i-1}\), with complement \(W_i\). If \(c_i\) is the complex codimension of \(S_i\) in \(W_{i-1}\), the Thom isomorphism for affine-space bundles gives
\[
 H^j(W_{i-1},W_i;k)\cong H^{j-2c_i}(S_i;k).
\]
Since the stratum \(S_i\) is an affine-space bundle over its fixed component \(F_i\),  its total Betti number is $\beta_k(F_i)$. Let \(\delta_i^j:H^j(W_i;k)\rightarrow H^{j+1}(W_{i-1},W_i;k)\) be the connecting map. Summing dimensions of the long exact sequence gives
\begin{equation}\label{eq:betti-defect}
 \beta_k(W_{i-1})=\beta_k(W_i)+\beta_k(F_i)-2\sum_j\rank_k\delta_i^j.
\end{equation}
Summing \eqref{eq:betti-defect} from \(i=1\) to $s$, using $W_0 = X, W_s =\emptyset$, and \eqref{eq:bb-betti} shows that all \(\delta_i^j\) vanish. Hence every restriction $H^j(W_{i-1};k)\rightarrow H^{j}(W_i;k)$ is onto, and in particular at $i=a$, \(H^\bullet(X;k)\to H^\bullet(Y;k)\) is onto. The commutativity of restriction and forgetful maps, together with Lemma~\ref{lem:projective-formality}, proves the second assertion.
\end{proof}

\vspace{1pt}

\subsection{Freeness of equivariant dualizing cohomology for proper varieties}
We will need freeness with dualizing coefficients on a possibly singular proper variety.

\begin{lemma}\label{lem:proper-duality}
Let \(X\) be a proper complex \(T\)-variety. If the $T$-action on $X$ is equivariantly formal over a field \(k\), then \(H_T^\bullet(X;\omega_{X,k})\) is finite graded free over \(P\).
\end{lemma}

\begin{proof}
Let \(\pi:X\to\pt\), and choose equivariant lifts of a homogeneous basis of \(H^\bullet(X;k)\). With \(b_q=\dim_k H^q(X;k)\), these classes give a morphism in the equivariant derived category \(D_T(\pt;k)\),
\begin{equation} \label{eq:mor in DC}
     \bigoplus_q k_{\pt}[-q]^{b_q}\longrightarrow R\pi_*k_X.
\end{equation}

After forgetting equivariance \eqref{eq:mor in DC} is an isomorphism on cohomology. The forgetful functor is conservative, so \eqref{eq:mor in DC} is an isomorphism in the equivariant derived category. Proper Verdier duality therefore gives
\[
 R\pi_*\omega_{X,k}\cong \mathbb{D}_T(R\pi_*k_X)\cong\bigoplus_q k_{\pt}[q]^{b_q}.
\]
Equivariant cohomology of the last direct sum is a direct sum of shifts of \(P\), as required.
\end{proof}

\vspace{1pt}

\subsection{Torsion, resolutions, and the proof of Theorem~\ref{thm:pd}}
For the rest of this section, let \(k=\Fp\). 

\begin{lemma}\label{lem:torsion-pullback}
Let \(X\) be an affine \(T\)-variety and let \(U\subset X\) be a $T$-invariant open subvariety. Suppose that \(H_T^\bullet(U;k)\) is \(P\)-torsion. Then \(H_T^\bullet(W;k)\) is \(P\)-torsion for every complex variety $W$ admitting a $T$-equivariant map
\(f: W\to U\).
\end{lemma}

\begin{proof}
Suppose that $T$ is of rank r and let $K\cong(S^{1})^r$ be the compact subtorus of $T$. Let \(K[p]\cong(\Z/p)^r\) be the full subgroup of elements of order dividing \(p\) in $K$. There are no \(K[p]\)-fixed points in \(U\). To see this, a nonzero polynomial in $P$ annihilates $1\in H_T^\bullet(U;k)$, because $H_T^\bullet(U;k)$ is $P$-torsion. If $u\in U$ were fixed by $K[p]$, the composite given by restriction to $u$
\[
P \to H_T^\bullet(U;k) \to H_{K[p]}^\bullet(U;k) \to H^\bullet(BK[p]; k)
\]
would be the standard restriction from the classifying space $BK$ to $BK[p]$. That restriction is injective. Indeed, for odd p, the degree-two Chern classes map to the polynomial generators of $H^\bullet(BK[p]; k)$; for $p= 2$, they map to the squares of independent degree-one generators. In either case the images are algebraically independent. However, a nonzero annihilator of $1$ contradicts this injectivity.

Choose a $T$-equivariant closed embedding of \(X\) into a finite-dimensional $T$-representation $V$, with coordinate eigenfunctions \(g_j\) and character weights \(\alpha_j\). The opens
\[
 U_j=\{g_j\ne0\}\cap U,\qquad \alpha_j\notin p\Lambda,
\]
cover \(U\): a point in $V$ for which every nonzero coordinate has weight in \(p\Lambda\) is fixed by \(K[p]\). Hence their inverse images \(W_j\) under $f$ cover \(W\). On \(W_j\), the pullback of the function \(g_j\) is a nowhere-zero section of the associated equivariant character line, so the image of \(\overline\alpha_j\in P\) in $H_T^2(W_j;k)$ is zero. Lift this class to the relative cohomology \(H_T^2(W,W_j;k)\) and take relative cup products.  Because $\cup_j W_j= W$, the product maps to zero in $H_T^\bullet(W;k)$.
Thus
\[
\left(\prod_j\overline\alpha_j\right)H_T^\bullet(W;k)=0.
\]
The product $\prod_j\overline\alpha_j$ is nonzero in the polynomial domain $P$. This proves the lemma.
\end{proof}

\begin{remark}\label{rem:finite-isotropy}
Unlike in characteristic $0$, the absence of \(T\)-fixed points alone does not imply \(P\)-torsion over \(\Fp\). For example, let \(T=\C^\times\) act on \(\C^\times\) by \(t\cdot z=t^p z\). Its equivariant cohomology is \(H^\bullet(B\mu_p;\Fp)\), on which the polynomial subring \(H^\bullet(BT;\Fp)\) acts injectively. The torsion hypothesis in Theorem~\ref{thm:pd} is essential to the following arguments.
\end{remark}

We now assume all the hypotheses of Theorem~\ref{thm:pd}.

\begin{lemma}\label{lem:resolution}
Under the hypotheses of Theorem~\ref{thm:pd}, there is a projective equivariant resolution \(f:Y\to X\) such that, for \(F=f^{-1}(x)\) and \(E=Y\setminus F\), there is an exact sequence
\begin{equation}\label{eq:resolution-presentation}
 0\longrightarrow H_T^\bullet(F;\omega_{F,k})\longrightarrow H_T^\bullet(Y;k)[2n]\longrightarrow H_T^\bullet(E;k)[2n]\longrightarrow0,
\end{equation}
whose first two terms are finite graded free over \(P\).
\end{lemma}

\begin{proof}
Embed \(X\) equivariantly into a $T$-representation $V$ and take its projective closure \(\overline X\subset \mathbb{P}(V\oplus \C)\). Equivariant resolution in characteristic zero gives a smooth projective \(T\)-variety \(Z\) with a projective birational map $Z\to \overline X$, which is an isomorphism over the smooth locus; one can use the canonical embedded resolution by invariant blow-up centers, see \cite{Wlodarczyk}. Let \(Y\) be the inverse image of \(X\). Then $Y$ is a smooth $T$-invariant open subvariety of $Z$, and $f : Y\to X$ is projective and birational.

Choose a cocharacter \(\lambda\) of $T$ contracting \(X\) to \(x\). Properness of \(f\) implies that every \(y\in Y\) has a limit as $t\to 0$ under $\lambda(t)$. The fixed locus \(Y^\lambda\) lies in the fiber \(F\), so it is proper. Lemma~\ref{lem:open-formality} shows that \(H_T^\bullet(Y;k)\) is finite graded free over $P$.

A smooth quasi-projective variety with proper $\lambda$-fixed locus and with all limits at $t\to 0$ is \emph{semiprojective}; its \emph{core} consists of the points whose limits at $t\to \infty$ also exist, see \cite{HRV}. Here the core is exactly \(F\). Indeed, every point of \(F\) has such a limit because \(F\) is projective. If \(f(y)\ne x\), a nonzero positive-weight coordinate of \(f(y)\) diverges at $t\to \infty$. 

Now, \cite[Theorem~1.3.1]{HRV} gives
\[ 
H^\bullet(Y;\Z)\xrightarrow{\sim}H^\bullet(F;\Z).
\]
This gives isomorphisms of cohomology groups with coefficients in \(k\) by the universal coefficient theorem, and comparison of the Leray--Serre spectral sequences of the Borel fibrations gives the equivariant isomorphism. Thus \(F\) is also equivariantly formal over $k$, and Lemma~\ref{lem:proper-duality} makes its dualizing equivariant cohomology finite graded free over $P$.

For \(i:F\hookrightarrow Y\) and \(j:E\hookrightarrow Y\), we have the closed-open triangle for the dualizing complex on the smooth variety $Y$
\[
 i_*\omega_{F,k}\longrightarrow k_Y[2n]\longrightarrow Rj_*k_E[2n]\longrightarrow.
\]
Lemma~\ref{lem:torsion-pullback}, applied to \(E\to U\), makes \(H_T^\bullet(E;k)\) torsion over \(P\). The connecting maps from $H_T^\bullet(E;k)[2n]$ to $H_T^\bullet(F;\omega_{F,k})[1]$ vanish because the source is $P$-torsion and the target is $P$-free. The long exact sequence therefore yields \eqref{eq:resolution-presentation}.
\end{proof}

We are ready to give a proof of Theorem~\ref{thm:pd}.

\begin{proof}[Proof of Theorem~\ref{thm:pd}]
Lemma~\ref{lem:resolution} gives \(\pd_P H_T^\bullet(E;k)\leq1\). Let \(f_U:E\to U\) be the restriction of the resolution, which is a proper birational morphism. Its unit and proper trace, using the canonical orientations on $E$ and on $U$, give equivariant morphisms
\begin{equation}\label{eq:trace}
 k_U\longrightarrow Rf_{U*}k_E \cong Rf_{U*}\omega_{E,k}[-2n] \longrightarrow\omega_{U,k}[-2n]\cong k_U.
\end{equation}
The composite is the identity over the dense smooth open locus where \(f_U\) is an isomorphism. Since \(U\) is connected,
\[
 \End_{D_T(U;k)}(k_U)=H_T^0(U;k)=k.
\]
Restriction to that open locus determines the scalar. The composite in \eqref{eq:trace} is therefore the identity everywhere.

It follows that \(C:=H_T^\bullet(U;k)\) is a graded $P$-module direct summand of \(H_T^\bullet(E;k)\), and hence also has projective dimension at most one. It is finite, nonzero, and torsion over the domain \(P\), so it cannot be projective. Thus \(\pd_PC=1\). At the homogeneous maximal ideal of $P$, the Auslander--Buchsbaum formula~\cite{AB} gives \(\depth C=r-1\). Since $C$ is torsion over $P$, a nonzero annihilator gives the Krull dimension \(\dim C\leq r-1\). Since the depth never exceeds the Krull dimension, we have \(\dim C= r-1\) and $C$ is Cohen--Macaulay.
\end{proof}

\section{Integral equivariant cohomology and specialization}\label{sec:specialization}

Return to the hypotheses of Theorem~\ref{thm:affine}. Fix a prime \(p\). In this section, we write
\[
 R=\Zp,\qquad F=\Qp,\qquad S=S_R,\qquad
 A=H_T^\bullet(U;R),\qquad B=H_T^\bullet(X;\omega_{X,R}).
\]
Flat change of coefficients and the hypothesis on integral equivariant cohomology imply that \(A\) is \(R\)-torsion-free. For an extension \(R\to R'\), subscripts on modules denote extension of scalars.

\subsection{Fundamental relations and Cohen--Macaulayness}
Let \(i:\{x\}\hookrightarrow X\) and \(j:U\hookrightarrow X\).

\begin{lemma}\label{lem:integral-sequence}
The modules \(A\) and \(B\) are concentrated in even degrees and are degreewise free over \(R\). There is an exact sequence of $S$-modules
\begin{equation}\label{eq:integral-sequence}
 0\longrightarrow S\xrightarrow{\varphi}B
 \longrightarrow A[2n]\longrightarrow0.
\end{equation}
If \(\mu:=\mu_X\) denotes the equivariant fundamental class, then we have the relations
\begin{equation}\label{eq:fundamental}
 B_F\cong S_F\mu, \qquad  A_F\cong S_F/(\pi), \qquad d\varphi(1)=\pi\mu,
 \end{equation}
where $d$ and $\pi$ are the numerator and denominator in the normalization \eqref{eq:intro-mult}, respectively. In particular, \(\mu\in B^{-2n}\) is a primitive generator and \(\pi A=0\).
\end{lemma}

\begin{proof}
The \(p\)-smoothness of \(U\) identifies the dualizing complex $\omega_{U,R}$ with $R_U[2n]$ and hence gives the triangle
\begin{equation}\label{eq:dualizing-triangle}
 i_*R_{\{x\}}\longrightarrow\omega_{X,R}
 \longrightarrow Rj_*R_U[2n]\longrightarrow.
\end{equation}
First we use coefficients in \(F\). Rational smoothness and equivariant contraction give
$$H_T^\bullet(X;\omega_{X,F})\cong H_T^\bullet(X;F)[2n] \cong S_F\mu, \quad \text{ with } \deg \mu=-2n.$$
By Lemma~\ref{lem:multiplicity}, point-pushforward is multiplication by \(\pi/d\) relative to \(\mu\). Hence after localization, the equality $ d\varphi(1)=\pi\mu$ holds over $F$. In particular, the point-pushforward is injective. The long exact sequence of the closed-open dualizing triangle therefore gives \(A_F\cong S_F/(\pi)\), concentrated in even degrees. Since \(A\) is \(R\)-torsion-free, it embeds into \(A_F\). Thus \(A\) is even and \(\pi A=0\).

Over \(R\), the parity just proved splits the long exact sequence into \eqref{eq:integral-sequence}. Hence \(B\) is also even and degreewise free, being an extension of two such modules. In degree \(-2n\), restriction identifies \(B^{-2n}\) with \(H_T^0(U;R)=R\); the fundamental class maps to its orientation generator, so $\mu$ is primitive. The equality $ d\varphi(1)=\pi\mu$ holds over \(F\), and therefore over \(R\), since $B$ is also $R$-torsion-free. 
\end{proof}

We need the following Cohen--Macaulayness result.

\begin{lemma}\label{lem:denominator}
One has \(\pi\bmod p\ne0\) in $S_{\Fp}=S/pS$. Moreover,
\[
 H_T^\bullet(U;\Fp)\cong A/pA
\]
is Cohen--Macaulay of dimension \(r-1\), and \(A\) is Cohen--Macaulay of dimension \(r\) at \(\m=(p,S^{+})\), where $S^{+}$ denotes the augmentation ideal generated by homogeneous elements of strictly positive degree.
\end{lemma}

\begin{proof}
If \(\pi=p\pi_1\), reducedness of \eqref{eq:intro-mult} forces the numerator \(d\in R^\times\). The relation \eqref{eq:fundamental} would give
\[
 \varphi(1)=p(d^{-1}\pi_1\mu)\in pB.
\]
But \(A[2n]\) in \eqref{eq:integral-sequence} is torsion-free over the PID \(R\) hence flat, so reduction modulo \(p\) preserves injectivity of the left map. The displayed equality contradicts this injectivity. Hence \(\overline\pi\ne0\).

The coefficient exact sequence and torsion-freeness in each degree identify \(H_T^\bullet(U;\Fp)\) with \(A/pA\). This module is torsion over $S_{\Fp}$ because by Lemma~\ref{lem:integral-sequence} it is annihilated by the nonzero polynomial \(\overline\pi\), so Theorem~\ref{thm:pd} applies. Finally, \(p\) is an \(A\)-nonzero-divisor. It follows that
\[
 \depth_{S_{\m}}A_{\m}
 =1+\depth_{(S/pS)_{(S^{+})}}(A/pA)_{(S^{+})}=r.
\]
The relation \(\pi A=0\) bounds the $S$-dimension of $A$ by \(r\). Depth and dimension must therefore both equal \(r\) and the claimed Cohen--Macaulay property follows.
\end{proof}

\vspace{1pt}

\subsection{Specialization to a generic line over an unramified extension}
The preceding lemma permits an integral specialization. All depth and parameter statements below may be checked at the homogeneous maximal ideal. They then hold for the graded modules themselves: a nonzero finite graded module over a polynomial ring over a local ring cannot vanish after localization at that ideal.

\begin{lemma}\label{lem:line}
There is a finite unramified extension \(\OO/R\) and an \(\OO\)-basis \(\ell_1,\ldots,\ell_{r-1},t\) of the degree-two part of \(S_{\OO}\) such that, with \(J=(\ell_1,\ldots,\ell_{r-1})\),
\begin{equation}\label{eq:line}
 S_{\OO}/J=\OO[t],\qquad \pi\bmod J=ut^n, \qquad u\in\OO^\times.
\end{equation}
The sequence \(\ell_1,\ldots,\ell_{r-1},p\) is regular on \(A_{\OO}\), and $\Abar=A_{\OO}/JA_{\OO}$ satisfies
\begin{equation}\label{eq:line-ranks}
\Abar^{2j}\cong\OO\ (0\leq j<n),\qquad \Abar^j=0\ \text{otherwise}.
\end{equation}
\end{lemma}

\begin{proof}
By Lemma~\ref{lem:denominator}, each factor \(\chi_i\) of \(\pi\) is nonzero modulo \(p\). Choose a finite extension \(k'/\Fp\) with \(|k'|>n\). The union of the \(n\) kernels of these linear forms contains at most \(n|k'|^{r-1}<|k'|^r\) vectors. Choose \(\bar v\) outside this union. Let \(\OO\) be the unramified extension of $R$ with residue field \(k'\), lift \(\bar v\) to a unimodular vector \(v\in\OO^r\), and complete it to a basis. Take dual coordinates \(\ell_1,\ldots,\ell_{r-1},t\) with 
\[
 \ell_i(v)=0,\qquad t(v)=1.
\]
Then \(\chi_i\bmod J=\chi_i(v)t\), with every \(\chi_i(v)\) a unit. This proves \eqref{eq:line}, with $u=\prod_i \chi_i(v)$.

Finite unramified extension preserves Cohen--Macaulayness. The quotient \(A_{\OO}/(p,J)A_{\OO}\) is finite-dimensional over $k'$: it is finite over \(k'[t]\) and is annihilated by \(t^n\). Thus \(p,\ell_1,\ldots,\ell_{r-1}\) is a system of parameters on the dimension-\(r\) Cohen--Macaulay module \(A_{\OO}\). Every ordering of a system of parameters on a Cohen--Macaulay module over a Noetherian local ring is regular~\cite[Tag~00N6 and Tag~00LJ]{Stacks}. This proves the claimed regularity.

By Lemma~\ref{lem:integral-sequence} and \eqref{eq:line}, the quotient \(\Abar\) is annihilated by \(t^n\), and is therefore finite over \(\OO\). It is a finitely generated torsion-free module over this discrete valuation ring, hence free. Over the finite unramified extension \(F'=\Frac(\OO)\) of $F=\Qp$, its rationalization is
\[
 \Abar\otimes_{\OO}F'\cong F'[t]/(t^n).
\]
This determines the ranks in \eqref{eq:line-ranks}. 
\end{proof}

\begin{remark}
Passing to a larger residue field can be necessary. For example, the kernels of \(x_1,x_2,x_1+x_2\) cover \(\F_2^2\). The new coordinates in Lemma~\ref{lem:line} are linear forms over \(\OO\); they need not be integral characters. We address this distinction here before forgetting equivariance in the next section.
\end{remark}

Now we have the following calculation.
\begin{lemma}\label{lem:lattices}
The sequence \(\ell_1,\ldots,\ell_{r-1},t\) is regular on \(B_{\OO}\). For \(\Bbar=B_{\OO}/JB_{\OO}\), there are nonzero \(c_0,\ldots,c_{n-1}\in\OO\) such that
\begin{equation}\label{eq:product}
 \prod_{k=0}^{n-1}c_k=d/u
\end{equation}
and a graded decomposition
\begin{equation}\label{eq:quotient}
 \Bbar/t\Bbar\cong \OO b_0\oplus\bigoplus_{j=0}^{n-1}\OO/(c_j)b_{j+1}, \text{ with } \deg b_j=-2n+2j, j=0,1, \dots, n.
\end{equation}
In particular,
\begin{equation}\label{eq:length}
 \length_{\OO}\tors_{\OO}(\Bbar/t\Bbar)=v_p(d).
\end{equation}
\end{lemma}

\begin{proof}
In a short exact sequence, a sequence regular on the two outer modules is regular on its middle module, and quotienting by the sequence preserves exactness. This follows one element at a time from the snake lemma. Applying it to \eqref{eq:integral-sequence}, after extending coefficients to $\OO$, and to the generators of \(J\), we obtain
\begin{equation}\label{eq:line-sequence}
0\longrightarrow\OO[t]\xrightarrow{\bar\varphi}\Bbar \longrightarrow\Abar[2n]\longrightarrow0.
\end{equation}
The degreewise freeness of the two outer modules makes the middle module \(\Bbar\) degreewise \(\OO\)-free; hence \(\Bbar\) embeds in its rationalization \(F'[t]\mu\). Multiplication by \(t\) is injective there, and hence on \(\Bbar\). This proves the regularity assertion.

By \eqref{eq:line-ranks} and \eqref{eq:line-sequence}, each graded piece \(\Bbar^{-2n+2j}\) is free of rank one for every \(j\geq0\), and all other graded pieces vanish. Choose its generator $b_j$ such that
\[
 b_0=\mu,\qquad b_n=\bar\varphi(1),\qquad b_j=t^{j-n}b_n\quad(j\geq n).
\]
The first choice uses primitivity of the orientation class. The other choices are possible because \(\Abar[2n]\) vanishes in nonnegative degrees and $\bar\varphi$ is an isomorphism in those degrees. Define the transition coefficients \(c_j\in \OO\setminus\{0\}\), for \(0\leq j<n\), by
\begin{equation}\label{eq:transition}
 tb_j=c_jb_{j+1}.
\end{equation}
For \(j\geq n\), the corresponding transition coefficient is one. The specialized fundamental relation~\eqref{eq:fundamental} gives
\[
 db_n=ut^nb_0.
\]
Iterating \eqref{eq:transition} proves \eqref{eq:product}. Taking the quotient by \(t\) degreewise gives \eqref{eq:quotient}. Since \(\OO/R\) is unramified, its normalized valuation still has \(v_p(p)=1\). Taking lengths in \eqref{eq:quotient} and using \eqref{eq:product} proves \eqref{eq:length}.
\end{proof}

\section{Forgetting equivariance and computing ordinary torsion}\label{sec:ordinary}
In this section, we give a proof of Theorem~\ref{thm:affine} and Corollary~\ref{cor:affine-integral}. 

\begin{lemma}\label{lem:koszul}
Let \(M\) be a finite graded module, bounded below, over \(\OO[z_1,\ldots,z_r]\), where the variables have positive degree. If an invertible \(\OO\)-linear change of the variables \(z_1,\ldots,z_r\) is an \(M\)-regular sequence, then \(z_1,\ldots,z_r\) is also an \(M\)-regular sequence.
\end{lemma}

\begin{proof}
An invertible change of generators identifies the Koszul complexes, including their exterior generators and differentials. Thus the Koszul complex on \(z_1,\ldots,z_r\) is acyclic in positive homological degrees. To recover regularity, view the full complex as the mapping cone of multiplication by \(z_r\) on the complex \(K'\) for the first \(r-1\) variables. The long exact sequence makes multiplication by \(z_r\) surjective on each \(H_i(K')\) for \(i>0\). A positive-degree map cannot be surjective on a nonzero bounded-below graded module. Hence \(H_i(K')=0\) for \(i>0\). Induction gives regularity of the first \(r-1\) variables, and vanishing of the full first homology gives injectivity of $z_r$ on $H_0(K') = M/(z_1,...,z_{r-1})M$.
\end{proof}

Now we compute $H^\bullet(X;\omega_{X,\OO})$ in the following two lemmas.
\begin{lemma}\label{lem:forget}
With the notation of Section~\ref{sec:specialization},
\begin{equation}\label{eq:forget}
 H^\bullet(X;\omega_{X,\OO}) \cong B_{\OO}/S_{\OO}^{+}B_{\OO} \cong\Bbar/t\Bbar.
\end{equation}
\end{lemma}

\begin{proof}
Take a basis \(z_1,\ldots,z_r\) of the integral character lattice. Lemmas~\ref{lem:lattices} and \ref{lem:koszul} show that it is a regular sequence on \(B_{\OO}\). Then \cite[Proposition~3.2]{JW} applies; it works over $\OO$ with the same proof. Explicitly, the character basis determines a tower of principal circle bundles from the Borel construction to the nonequivariant space. In each Gysin sequence, the Euler-class map is multiplication by the corresponding character. Its injectivity identifies cohomology at the next stage with the cokernel. Apply this successively using regularity. Gysin sequences with the pulled-back dualizing complex are obtained in the same way from the circle-bundle projection formula.  Applying it to the equivariant dualizing complex gives \eqref{eq:forget}.
\end{proof}

\begin{lemma}\label{lem:endpoints}
There are isomorphisms
\begin{equation*}
    \begin{aligned}
 H^j(X;\omega_{X,\OO})&\cong H^{j+2n}(U;\OO), \quad j<-1,\\
H^0(X;\omega_{X,\OO})&=0,\\
H^{2n-1}(U;\OO)&\cong\OO.
\end{aligned}
\end{equation*}
In particular, the coefficient \(c_{n-1}\) of Lemma~\ref{lem:lattices} is a unit.
\end{lemma}

\begin{proof}
Apply ordinary cohomology to the closed-open triangle \eqref{eq:dualizing-triangle}, with coefficients in \(\OO\). Since the point has cohomology only in degree zero, the long exact sequence gives the first isomorphism. Lemma~\ref{lem:link} gives
\[
 H^0(X;\omega_{X,\OO})=H_0^{\BM}(X;\OO)=0.
\]
By \eqref{eq:forget} and \eqref{eq:quotient}, \(H^{-1}(X;\omega_{X,\OO})=0\) as well. The relevant part of the long exact sequence is therefore
\[
0\longrightarrow H^{2n-1}(U;\OO)\longrightarrow\OO \longrightarrow0.
\]
Finally, the only degree-zero summand of the right hand side of \eqref{eq:quotient} is \(\OO/(c_{n-1})b_n\). By \eqref{eq:forget}, its vanishing proves that \(c_{n-1}\) is a unit.
\end{proof}

We are ready to give a proof of Theorem~\ref{thm:affine} and Corollary~\ref{cor:affine-integral}.

\begin{proof}[Proof of Theorem~\ref{thm:affine}]
Lemmas~\ref{lem:lattices}, \ref{lem:forget} and \ref{lem:endpoints} identify all torsion terms and give
\[
\length_{\OO}\tors_{\OO}H^\bullet(U;\OO)=v_p(d).
\]
For a finite torsion \(R\)-module \(M\), extension to \(\OO\) preserves length: each summand \(R/(p^a)\)  in its invariant factor decomposition becomes \(\OO/(p^a)\), of length \(a\). The residue-field degree of $\OO$ changes cardinality over $\OO$ but does \emph{not} change this length. By Lemma \ref{lem:link}, finite CW type gives flat base change
\[
H^\bullet(U;\OO)=H^\bullet(U;R)\otimes_R\OO.
\]
Hence we obtain \eqref{eq:affine-order}.

The torsion summand \(\OO/(c_k)b_{k+1}\) in \eqref{eq:quotient} has degree \(2k+2\) in $H^\bullet(U;\OO)$. The last such summand is zero by Lemma~\ref{lem:endpoints}. Thus every intermediate odd group vanishes, and there is at most one cyclic summand in each intermediate even degree. These assertions descend from $\OO$ to \(R\) by faithful flatness and the invariant-factor classification over a discrete valuation ring. Put \(a_i=v_p(c_{i-1})\), for \(1\leq i<n\). The product identity and the unit \(c_{n-1}\) give \(\sum_i a_i=v_p(d)\). The assertion for degree $2n-1$ and the vanishing in higher degrees follow from Lemmas~\ref{lem:link} and \ref{lem:endpoints}. This proves \eqref{eq:degreewise}.
\end{proof}

\begin{proof}[Proof of Corollary~\ref{cor:affine-integral}]
The integral hypotheses imply the hypotheses of Theorem~\ref{thm:affine} at every prime. Ordinary integral cohomology is finitely generated by Lemma~\ref{lem:link}, and flat change to \(\Zp\) identifies its \(p\)-primary torsion with the torsion of \(H^\bullet(U;\Zp)\). Multiplying the orders over all primes proves \eqref{eq:integral-order}. Each intermediate even group has cyclic $p$-primary part at every prime $p$, so it is cyclic. The intermediate odd groups vanish after tensoring with every \(\Zp\), and hence vanish.
\end{proof}

\begin{remark}
The numerator $d$ determines the order of the total torsion, but the proof does not assert that it determines each cohomology group. The individual valuations of the transition coefficients in \eqref{eq:transition} contain the finer degreewise information.
\end{remark}

We record the following duality result when the punctured \(U\) is smooth.

\begin{corollary}\label{cor:square}
Under the hypotheses of Corollary~\ref{cor:affine-integral}, assume that \(U\) is smooth. For \(1\leq i<n\), set \(m_i=|H^{2i}(U;\Z)|\), taking the order of the zero group to be one. Then
\[
 m_i=m_{n-i},\qquad |d|=\prod_{i=1}^{n-1}m_i.
\]
In particular, \(|d|\) is a square if \(n\) is odd. The same conclusion holds for \(f_{v,w}\) when the punctured Schubert slice $U$ in Theorem~\ref{thm:schubert} is smooth.
\end{corollary}

\begin{proof}
The link $L$ in Lemma \ref{lem:link} is then a closed oriented smooth manifold of dimension \(2n-1\). Its nonsingular torsion linking pairing pairs the torsion in \(H^j\) with that in \(H^{2n-j}\). Thus \(m_i=m_{n-i}\). Corollary~\ref{cor:affine-integral} gives the product formula.
\end{proof}

\section{Applications to Schubert varieties}\label{sec:schubert}

We explain the geometric hypotheses in Theorem~\ref{thm:schubert} and record an alternative parity sheaf argument. The latter can replace Section~\ref{sec:geometry} when one is interested only in Schubert slices.

\subsection{Equivariant torsion-freeness and the Schubert torsion formula}
Retain the notation of the introduction. The standard normal slice $N$ is affine and \(T\)-stable, with an attractive point \(x_v\), and there is a product neighborhood (see~\cite[Section 7]{JW})
\begin{equation}\label{eq:slice-product}
 C_v\times N\xrightarrow{\sim}\widetilde N\subset X_w.
\end{equation}
This neighborhood is compatible with the Bruhat stratification. The word normal in ``normal slice'' refers to transversality. We will use the following result due to Fiebig and Williamson (see~\cite[Corollary 8.10]{FW}); see also \cite[Proposition 7.1]{JW}. Its proof uses parity sheaves~\cite{JMW}.

\begin{proposition}\label{prop:schubert-free}
If \(N\setminus\{x_v\}\) is \(p\)-smooth, then \(H_T^\bullet(N\setminus\{x_v\};\Z)\) has no \(p\)-torsion.
\end{proposition}

At a rationally smooth point, Kumar's criterion gives
\[
 e_{x_v}N=\frac{f_N}{\pi_N},
\]
where \(\pi_N\) is a product of \(n\) primitive root characters and \(f_N\) is a positive integer, with consistent signs for the weights; see \cite{Kumar,JW}. Multiplicativity in \eqref{eq:slice-product} gives
\begin{equation}\label{eq:numerator}
 e_{x_v}X_w=\frac{e_{x_v}N}{\det(T_{x_v}C_v)}=\frac{f_N}{\pi_N\det(T_{x_v}C_v)}.
\end{equation}
Every factor in this denominator is primitive in \(\Z\Phi\). By Gauss's lemma the denominator has content one. A constant numerator cannot cancel a nonconstant factor, nor can it cancel an
integer factor. Consequently \(f_N=f_{v,w}\).

\begin{proof}[Proof of Theorem~\ref{thm:schubert}]
If $v = w$, then $U$ is empty and $f_{w,w}= 1$, so both torsion-order formulas are immediate. For a positive-dimensional normal slice $N$, Proposition~\ref{prop:schubert-free} supplies the coefficient hypothesis of Theorem~\ref{thm:affine}. Equation~\eqref{eq:numerator} identifies its numerator with \(f_{v,w}\). Hence the local assertion follows. If \(U\) is $\Z$-smooth, one can apply the local assertion at every prime and use finite generation, as in the proof of Corollary~\ref{cor:affine-integral}.
\end{proof}

\vspace{1pt}

\subsection{Cohen--Macaulayness via parity sheaves}
For completeness, we give a direct proof of the regularity needed in Section~\ref{sec:specialization} using parity sheaves. Put \(R=\Zp\) and \(S=S_R\). Let \(\cP_w\) be the indecomposable $T$-equivariant \emph{parity sheaf} (with respect to the natural \emph{constant pariversity}) of the Schubert variety \(X_w\), normalized such that its restriction to the Schubert cell \(C_w\) is the $T$-equivariant constant sheaf in degree zero. Existence of \(\cP_w\) and the stalk and costalk properties follow from the theory of parity sheaves on Schubert varieties \cite{JMW,FW}.

\begin{proposition}\label{prop:parity-presentation}
Suppose that \(U:=N\setminus\{x_v\}\) is \(p\)-smooth. Then \(A=H_T^\bullet(U;R)\) admits an exact sequence
\begin{equation}\label{eq:parity-presentation}
 0\longrightarrow F_1\longrightarrow F_0 \longrightarrow A\longrightarrow0
\end{equation}
with \(F_0,F_1\) finite graded free over \(S\). If $\dim_{\C}N>0$ and \(N\) is also rationally smooth, then \(A\) is Cohen--Macaulay of dimension \(r\), the rank of the maximal torus $T$.
\end{proposition}

\begin{proof}
If $v = w$, then $A = 0$ and \eqref{eq:parity-presentation} is immediate. We henceforth assume $v < w$.

The normal slice $N$ meets every cell \(C_y\) with \(v\leq y\leq w\). Indeed, \(x_v\in\overline{C_y}\), so the open neighborhood \(\widetilde N\) meets \(C_y\); the \(C_v\)-coordinate can be removed by the action of the Borel group $B$ in \eqref{eq:slice-product}. Conversely, contraction of a point in \(N\cap C_y\) to \(x_v\) implies \(v\leq y\). Also \(N\cap C_v=\{x_v\}\). Thus, if
\[
 Z:=\bigcup_{v<y\leq w}C_y,
\]
then \(Z\) is open in \(X_w\) and \(N\cap Z=N\setminus\{x_v\}\).

For every \(y>v\), choose a point in \(N\cap C_y\). The product neighborhood transfers \(p\)-smoothness of the slice at that point to \(X_w\). Transitivity of the Borel action on \(C_y\) shows that every point of \(C_y\) is \(p\)-smooth. Hence \(Z\) is \(p\)-smooth.

The $T$-equivariant constant sheaf $R_Z$ on $Z$ is parity: stalk restrictions are even and free, and duality gives the corresponding costalk assertion. It is the indecomposable parity extension of the $T$-equivariant constant sheaf on \(C_w\). The restriction and uniqueness theorems for parity sheaves \cite[Proposition~2.11 and Theorem~2.12]{JMW} imply \(\cP_w|_Z\cong R_Z\).

Set \(\cP:=\cP_w|_N\), with inclusions \(i:\{x_v\}\hookrightarrow N\), \(j:U\hookrightarrow N\). Then \(\cP|_U\cong R_U\). Let
\[
 F_1=H_T^\bullet(N;i_*i^!\cP),\qquad F_0=H_T^\bullet(N;\cP).
\]
Contraction identifies \(F_0\) with equivariant cohomology of the stalk at \(x_v\). For \(F_1\), the transverse product \eqref{eq:slice-product} identifies \(i^!\cP\) with the restriction at \(x_v\) of the costalk along \(C_v\). Both equivariant cohomologies are finite free and even by parity. This is also the slice calculation used in \cite[Proposition~8.9]{FW}.

Apply equivariant cohomology to the distinguished triangle
\[
i_*i^!\cP\longrightarrow\cP\longrightarrow Rj_*R_U \longrightarrow.
\]
After extension to the fraction field of \(S\), the map \(F_1\to F_0\) is an isomorphism by localization. Its kernel is therefore \(S\)-torsion, and must be zero because \(F_1\) is free. The long exact sequence gives \eqref{eq:parity-presentation}.

Now assume the rational smoothness of $N$. Coefficient-torsion-freeness is already supplied by Proposition~\ref{prop:schubert-free}. Now Lemma~\ref{lem:integral-sequence} gives \(\pi A=0\), so \(\dim A_{\m}\leq r\) at the maximal ideal $\m=(p,S^{+})$ of $S$. The depth lemma~\cite[Tag~00LX]{Stacks} applied to \eqref{eq:parity-presentation}, over the regular local ring \(S_{\m}\) of dimension \(r+1\), gives \(\depth A_{\m}\geq r\). Thus both \(\dim A_{\m}\) and \(\depth A_{\m}\) equal \(r\), which proves the Cohen--Macaulayness. 
\end{proof}

The elementary denominator argument at the start of Lemma~\ref{lem:denominator} and the rest of Sections~\ref{sec:specialization}--\ref{sec:ordinary} now prove the Schubert torsion formula directly.

\end{document}